\documentclass[final]{amsart}
\usepackage{amsmath,amsfonts,amssymb,amscd,verbatim,multicol}
\usepackage{commands}
\usepackage{stmaryrd}
\usepackage[T1]{fontenc}
\usepackage{lmodern}
\usepackage{mathrsfs}

\newcommand{\aq}{\mathcal{A}_q}
\newcommand{\ap}{\mathcal{A}_p}
\newcommand{\apq}{\mathcal{A}_{p,q}}
\newcommand{\hpq}{H_{p,q}}
\newcommand{\hwa}{H(\wa)}

\newcommand{\uh}{U(\mathfrak{h})}

\newcommand{\aprs}{A_p(r,s)}
\newcommand{\apn}{A_p^n(r,s)}
\newcommand{\vpq}{\mathcal{V}_{p,q}}
\newcommand{\vq}{\mathcal{V}_q}

\newcommand{\sH}{\mathscr{H}}

\newcommand{\ox}{\bar x}
\newcommand{\oy}{\bar y}
\newcommand{\oz}{\bar z}

\newcommand{\mH}{\mathcal H}
\newcommand{\bI}{\mathbb I}

\title{Two-parameter analogs of the Heisenberg enveloping algebra}
\author{Jason Gaddis}
\address{Department of Mathematics, UCSD, La Jolla, CA 92093-0112, USA}
\email{jgaddis@ucsd.edu}
\date{}

\begin{document}

\begin{abstract}
One-parameter analogs of the Heisenberg enveloping algebra were studied 
previously by Kirkman and Small. 
In particular, they demonstrated how one may obtain Hayashi's 
analog of the Weyl algebra as a primitive factor of this algebra.
We consider various two-parameter versions of this problem.
Of particular interest is the case when the parameters are dependent.
Our study allows us to consider the representation theory of a
two-parameter version of the Virasoro enveloping algebra.
\end{abstract}

\maketitle

\section{Introduction}
\label{intro}

The Heisenberg Lie algebra $\mathfrak{h}$ is defined on generators 
$\{x,y,z\}$ subject to the relations 
$[x,y]=z$, $[x,z] = [y,z] = 0$.
The first Weyl algebra is a simple factor ring of 
the enveloping algebra $U(\mathfrak{h})$.
Determining the simple factor rings for enveloping algebras and 
quantum enveloping algebras is a long-standing problem in algebra 
and representation theory.

Kirkman and Small introduced a $q$-analog to $U(\mathfrak{h})$ 
where the commutator is replaced by the `quommutator' \cite{kirksmall}.
A key result of theirs is that there exist a degree two central element 
$\Omega$ such that $U_q(\mathfrak{h})/(\Omega-1)U_q(\mathfrak{h})$ 
is isomorphic to Hayashi's $q$-analog of the Weyl algebra
\cite{hayashi}, 
hence providing a simple factor ring in the quantum setting.

We generalize this situation further to the case of two parameters $p,q \in \kk^\times$ and study the algebra
	\[ \hpq = \kk\langle x,y,z \mid  zx-p\inv xz, zy-pyz, yx-q xy-z \rangle.\]
We explore various constructions of $\hpq$ in Section \ref{construct} and show 
that it is 3-dimensional Artin-Schelter regular (Proposition \ref{props1}) 
and a prime noetherian domain (Proposition \ref{prop.skpoly}).

We study the prime spectrum of $\hpq$,
with particular interest given to 
when (and how) $\hpq$ gives rise to a
simple factor ring analogous to the Weyl algebra.
There are three cases to consider corresponding to various relations between
the parameters $p$ and $q$.
We denote these cases in the following way,
\begin{itemize}
	\item[] Section \ref{sec.inv}, Inverse parameter case: 
		$p=q\inv$;
	\item[] Section \ref{sec.indep}, Independent parameter case: 
		$p^r \neq q^s$ for all $r,s \in \ZZ$;
	\item[] Section \ref{sec.dep}, Dependent parameter case: 
		$p^r=q^s$ for some $r,s \in \ZZ_+$ with $\gcd(r,s)=1$. 
\end{itemize}

The inverse parameter case was considered by Kirkman and Small.
We extend their results on prime ideals to $p,q$ roots of unity
(Proposition \ref{inv.primes}).

Benkart has studied multiparameter Weyl algebras in the independent
parameter case \cite{benkweyl}.
We demonstrate how to obtain Benkart's algebras from $\hpq$
through the use of noncommutative dehomogenization.
We then give an independent proof of the ring-theoretic properties of this algebra (Proposition \ref{indep.props}).

Of particular interest in this work is the \textit{dependent parameter case}.
In this case, assuming $p$ and $q$ are nonroots of unity,
$\hpq$ has a simple factor ring (Theorem \ref{simple}) which 
generalizes the Hayashi-Weyl algebra.
We give defining relations for these algebras
and their higher dimensional analogs, as well as determine a 
faithful, irreducible representation (Proposition \ref{vn.irred}).

Dean and Small \cite{deansmall} have given a method for finding 
representations of the Virasoro algebra from those of the Weyl algebra. 
Kirkman and Small extended this to the one-parameter case.
In Section \ref{sec.vira}, we study a two-parameter analog of the
Virasoro algebra in a similar manner.

\section{Background}
\label{bkgrd}

As explained in \cite{kirksmall}, 
the harmonic oscillator problem in quantum mechanics is to find operators 
$a$ and $a^+$ acting on a Hilbert space with orthonormal basis 
$\{v_n\}_{n \geq 0}$ such that $aa^+-a^+a = 1$ and 
$\hat{H}v_n = (n+1/2)\hbar\omega v_n$, 
where $\hat{H}=\hbar\omega (a^+a^-+1/2)$ is the Hamiltonian\footnote{$\hbar$ is 
Planck's constant and $\omega$ is angular momentum.}. 
The operators $a$ and $a^+$ are typically referred to as 
the annihilation and creation operators, respectively. 
From an algebraic viewpoint, 
the study of the harmonic oscillator problem is equivalent to studying 
the representation theory of the Weyl algebra, 
\[\fwa = \kk\langle x,y \mid yx-xy-1 \rangle.\]

Suppose our space is $\mathbb{R}^n$ and 
the $v_n$ are the standard basis vectors. 
We can represent $a$ and $a^+$ as matrices
\begin{align}\label{mrep}
  a = \begin{pmatrix}  0 & \sqrt{1} &     0       & 0 & 0 & \cdots \\
            0 & 0  &      \sqrt{2}   & 0 & 0 & \cdots \\
            0 & 0  &      0      & \sqrt{3} & 0 & \cdots \\
            \vdots & \vdots & \vdots    & \vdots  & \ddots & \vdots\end{pmatrix}, ~
  a^+ = \begin{pmatrix}0 & 0 &     0       & 0 & 0 & \cdots \\
            \sqrt{1} & 0  &      0   & 0 & 0 & \cdots \\
            0 & \sqrt{2}  &      0      & 0 & 0 & \cdots \\
            \vdots & \vdots & \ddots    & \vdots  & \vdots & \vdots\end{pmatrix}.
\end{align}
Then $a^+v_n = \sqrt{n+1}v_{n+1}$ for $n \geq 0$ and 
$av_0=0, av_n = \sqrt{n}v_{n-1}$ for $n \geq 1$.

Define the \textit{$q$-number} to be
\begin{align}
\label{qnum}
    [n]_q = \frac{q^n-q^{-n}}{q-q\inv} = \sum_{i=0}^{n-1} q^{2i}.
\end{align}
Let $N$ be the diagonal matrix with 
$n_{ii}=[i]_q$ for $i \geq 0$ (the number operator).
If we replace $\sqrt{n}$ with $\sqrt{[n]_q}$ in \eqref{mrep}, then
the matrices $a,a^+$ satisfy the relations
$aa^+-qa^+a = q^{-N}$, $Na^+-a^+N = a^+$, and $Na-aN = -a$. 
In terms of the standard basis vectors $v_n$,
the operators satisfy
\begin{align*}
  a^+v_n &= \sqrt{[n+1]_q}v_{n+1} \text{ for } n \geq 0, \\
  av_0 &= 0 \text{ and } 
  av_n = \sqrt{[n]_q}v_{n-1} \text{ for } n \geq 1.
\end{align*} 
In this way, these operators may be regarded as a 
$q$-analog of the harmonic oscillator. 
Making the replacement $L=q^{-N}$ so that 
$La^+ = q\inv a^+L$ and $La=qaL$, as in \cite{kirksmall},
gives the relations for the algebra $H_q$, 
a $q$-analog of $U(\mathfrak{h})$. 
The matrix representation above is indeed an irreducible 
representation of $H_q$, but it is not faithful. 
Let $\Omega = (aa^+ - (1-q^2)L)L$. 
Then $\Omega$ is a central element of $H_q$ and 
$H_q/(\Omega-1)H_q =: \aq$ is the Hayashi $q$-analog of the 
Weyl algebra \cite{hayashi}. 
Then \eqref{mrep} is a faithful, irreducible representation of $\aq$.

There is another analog of the Weyl algebra that is well-studied in 
ring theory, namely the \textit{quantum Weyl algebra}.
For $q \in \kk^\times$, this algebra has presentation 
\begin{align}
\label{qwa.def}
	\wa = \kk\langle x,y \mid yx-q xy-1 \rangle.
\end{align}
In higher dimensions, quantum Weyl algebras can be constructed 
as differential operators on quantum affine $n$-space. 
However, the ring theoretic properties of 
$\wa$ do not mimic those of $\fwa$ well. 
In particular, $\fwa$ is a simple noetherian domain of 
global dimension one and GK dimension two.
While $\wa$ is a noetherian domain with GK dimension two,
it is not simple and has global dimension two.
On the other hand, $\aq$ shares with $\fwa$ all of the 
above mentioned properties.

Analogously, Chakrabarti and Jagannathan \cite{chakjag} have considered 
two-parameter analogs of $\mathfrak{su}(2)$ to construct 
$(p,q)$-oscillators satisfying
\begin{align*}
  aa^+-qa^+a = p^{-N}, \;\;\;
  aa^+-p\inv a^+a = q^N, \;\;\; 
  Na^+-a^+N = a^+, \;\;\;
  Na-aN = -a.
\end{align*}
Define the \textit{$(p,q)$-number} to be
\begin{align}
\label{pqdef}
  [n]_{p,q} = \frac{q^n-p^{-n}}{q-p\inv} = \sum_{i=0}^{n-1} q^i p^{-(n-i)}.
\end{align}
When $p=q$, this reduces to the standard $q$-number.
Note that $[n]_{p,q}=0$ if $p$ and $q$ are both primitive roots of 
unity and both orders divide $n$. 
We obtain a matrix representation of these operators by making the 
replacement $\sqrt{[n]_{p,q}}$ for $\sqrt{n}$ in \eqref{mrep} and 
noting the identities
\begin{align}
\label{pqrel}
  [n+1]_{p,q} 
  	= p^{-n} + q[n]_{p,q} 
  	= q^n + p\inv[n]_{p,q}.
\end{align}

Suppose $q^r=p^s$ for some integers $r$ and $s$. 
The identification $L=p^{-N}$ gives $q^{rN}=L^{-s}$ 
and so we have the following relations
\begin{align}
\label{aqn-rel}
  aa^+-q a^+a= L, \;\;\;
  (aa^+-p\inv a^+a)^r = L^{-s}, \;\;\;
  La^+-p\inv a^+L = 0, \;\;\;
  La-paL = 0.
\end{align}
We call this algebra $\aprs$. 
Like $\aq$, it is simple and the corresponding representation 
is faithful and irreducible. 
Further study of this algebra is contained in Section \ref{sec.dep}.

\section{Constructions, identities, and basic properties}
\label{construct}

Throughout, $\kk$ is an uncountable, algebraically closed, 
characteristic zero field.

We begin by considering multiple constructions of $\hpq$. 
Throughout, we fix $p,q \in \kk^\times$.
Initially, we place no restrictions on roots of unity or 
relation between the parameters.
As a $\kk$-algebra, we have the following presentation
\begin{align}
\label{reln.hberg}
	\hpq = \kk\langle x,y,z \mid zx-p\inv xz, zy-pyz, yx-q xy-z \rangle.
\end{align}
Assigning degree one to $x$ and $y$ and degree two to $z$ gives 
$\hpq$ the form of a connected graded algebra.
In \cite{hayashi} and \cite{benkweyl},
the third defining relation is taken to be $yx-q^2xy=z^2$
but we opt for the above convention as in \cite{kirksmall}.

\begin{prop}
\label{props1}
The algebra $\hpq$ is (Artin-Schelter) regular of global and 
GK dimension three.
\end{prop}

\begin{proof}
Since $z$ is a central regular element in $\hpq$, 
then $\hpq$ is regular if and only if 
$\hpq/z\hpq$ is regular \cite{leva}. 
But $\hpq/z\hpq$ is isomorphic to the quantum plane 
$\mathcal{O}_{q}(\kk^\times)$,
which is regular of global and GK dimension two \cite{artsch}.
\end{proof}

For an $\NN$-graded algebra $A = \oplus_{n \in \NN} A_n$ and 
$\tau$ a graded automorphism of $A$, 
a \textit{Zhang twist} \cite{zhang} of $A$ is the algebra with the same 
$\kk$-algebra basis as $A$ but a new multiplication $*$ given by 
$f*g = f\tau^n(g)$ for all $f \in A_n$, $g \in A_m$.
We denote this algebra $A^\tau$. 

For $H=\hpq$, we define $\tau \in \Aut(H)$ by 
$\tau(x)=\sqrt{p} x$, $\tau(y) = \sqrt{p\inv} y$ and $\tau(z)=z$. 
Then in $H^\tau$ we have,
\begin{align*}
  &z*x-x*z = z\tau^2(x) - x\tau(z) 
  		= pzx-xz = p(zx-p\inv xz) = 0, \\
  &z*y-y*z = z\tau^2(y) - y\tau(z) 
  		= p\inv zy - yz = p\inv (zy-pyz) = 0, \\
  &y*x - pq x*y - z 
  	= y\tau(x) - pq x\tau(y) - z = \sqrt{p}yx - \sqrt{p}q xy - z
    = \sqrt{p}(yx- q xy - \sqrt{p} z).
\end{align*}
Recall the definition for the quantum Weyl algebra \eqref{qwa.def}.
The homogenized quantum Weyl algebra is the $\kk$-algebra
with parameter $q \in \kk^\times$ and presentation
	\[ \hwa = \kk\langle x,y,z \mid zx-xz, zy-yz, yx-qxy-z\rangle.\]
It is clear that $H^\tau \iso H(A_1^{pq}(\kk))$ via the map 
$x \mapsto x$, $y \mapsto y$, and $z \mapsto \sqrt{p\inv} z$.
Hence, the category of graded modules on $\hpq$ is equivalent
to that of $H(A_1^{pq}(\kk))$ \cite[Theorem 1.1]{zhang}.

\begin{remark}
Suppose we replace the relation $zx-p\inv xz$ in \eqref{reln.hberg}
with $zx-p'xz$ for some $p' \in k^\times$.
That is, one might hope for a three-parameter analog of $\uh$.
However, a direct applicaton of Bergman's Diamond Lemma \cite{diamond} 
shows that $\GK \hpq =3$ if and only if $p'=p\inv$.
\end{remark}

Let $R$ be a ring. 
Given $\sigma \in \Aut(R)$, a $\kk$-linear map 
$\delta:R \rightarrow R$ is said 
to be a \textit{$\sigma$-derivation} if it satisfies 
$\delta(r_1r_2)=\sigma(r_1)\delta(r_2)+\delta(r_1)r_2$ 
for all $r_1,r_2 \in R$. 
The \textit{skew polynomial ring} $S=R[x;\sigma,\delta]$ is the 
overring of $R$ with commutation given by $xr = \sigma(r)x+\delta(r)$ 
for all $r \in R$. 
If $\delta=0$ or $\sigma = \mathrm{id}_R$, 
then we abbreviate $S$ as $R[x;\sigma]$ or $S=R[x;\delta]$, respectively.

We say $\sigma$ is an \textit{inner automorphism} if there exists 
a unit $a \in R$ such that $\sigma(r)=a\inv ra$ for all $r \in R$. 
In this case, $R[x;\sigma,\delta] = R[ax;a\delta]$. 
We say $\delta$ is an \textit{inner $\sigma$-derivation} if there 
exists $t \in R$ such that $\delta(r)=tr-\sigma(r)t$ for all $r \in R$. 
In this case, $R[x;\sigma,\delta]=R[x-t;\sigma]$.

\begin{prop}
\label{prop.skpoly}
The algebra $\hpq$ may be presented as an 
iterated skew polynomial ring,
and hence is a prime, noetherian domain.
\end{prop}

\begin{proof}
Let $R=\kk[z][x;\alpha]$ where $\alpha$ is the $\kk[z]$-automorphism 
defined by $\alpha(z)=pz$,
so $R$ is a \textit{quantum plane}, $\qpp$. 
Then $\hpq \iso R[y;\sigma,\delta]$ where $\sigma$ is the 
$R$-automorphism given by $\sigma(x)=qx$, $\sigma(z)=p\inv z$ and 
$\delta$ is a $\sigma$-derivation on $R$ given by $\delta(x)=z, \delta(z)=0$.
\end{proof}

When $pq \neq 1$, 
$\hpq$ has the form of an ambiskew polynomial ring as defined in \cite{jordan}. 
The base ring in this case is $A=\kk[z]$. 
Let $u=(1-pq)\inv z \in A$ and let $\alpha$ be as in the 
skew polynomial construction. 
Then $xz=\alpha(z)x$, $yz=\alpha\inv(z)x$ and $yx-q xy=u-q\alpha(u)=z$. 
Certain aspects of this paper have been considered in \cite{jorwells} 
from the viewpoint of ambiskew polynomial rings.

Given a ring $D$, $\rho \in \Aut(D)$, and $a \in \cnt(D)$,
the \textit{Generalized Weyl Algebra} (GWA) $D(\rho,a)$ is defined to
be the ring generated by $D$ and indeterminates $x$ and $y$ satisfying
\begin{align}
\label{defn.gwa} 
	xd = \rho(d)x, \;\;
	yd = \rho\inv(d)y, \;\;
	xy = \rho(a), \;\;
	yx = a,
\end{align}
for all $d \in D$ \cite{bavjor}.
Let $D=\kk[yx,z]$ with $\rho \in \Aut(D)$ defined by
$\rho(z)=pz$ and $\rho(yx)=q\inv(yx-z)$.
An easy check shows $\hpq \iso D(\rho,yx)$.

Substituting $z=yx-qxy$ into the first two relations in \eqref{reln.hberg}
gives the algebra on generators $x$ and $y$ subject to the relations
\begin{align}
\label{du.def}
  yx^2 = (p\inv + q)xyx - p\inv q x^2y, ~~~ 
  y^2x = (p\inv + q)yxy - p\inv q xy^2.
\end{align}
Thus, $\hpq$ is a down-up algebra as defined in \cite{benkroby} with parameters 
$d=y$, $u=x$, $\alpha=q+p\inv$, $\beta = -p\inv q$, and $\gamma=0$. 
By \cite[Theorem 1.1]{kirkmus}, 
such an algebra has Hopf structure if and only if $\alpha + \beta = 1$. 
In our case, this reduces to the condition that $p=1$ or $q=1$.

Just as $U(\mathfrak{h})$ appears as a subalgebra of $U(\mathfrak{sl}_3)$, 
so too does $U_q(\mathfrak{h})$ appear as a subalgebra of $U_q(\mathfrak{sl}_3)$ \cite{kirksmall}. 
Benkart and Witherspoon \cite{benkspoon} have defined a two-parameter analog 
$U_{r,s}(\mathfrak{sl}_3)$ of $U(\mathfrak{sl}_3)$. 
The generators $\{e_1,e_2,e_3\}$ of $U_{r,s}(\mathfrak{sl}_3)$ 
satisfy the relations
\begin{align*}
 e_i^2 e_{i+1} - (r+s)e_i e_{i+1} e_i + rs  e_{i+1} e_i^2 &= 0, \\
 e_i e_{i+1}^2 - (r+s)e_{i+1} e_i e_{i+1} + rs  e_{i+1}^2 e_i &= 0.
\end{align*}
From \eqref{du.def}, it is clear that $\hpq \iso \kk\{e_i,e_j\}$, 
$i \neq j$, with $r=p\inv$, $s=q$.

\begin{lem}
\label{idlm1}
In $\hpq$ the following identities hold for $n > 0$,
\begin{align}
    \label{ident1} yx^n &= q^{n}x^ny + [n]_{p,q} x^{n-1} z, \\
    \label{ident2} y^nx &= q^{n}xy^n + [n]_{p,q} z y^{n-1}.
\end{align}\end{lem}
  
\begin{proof}
We prove \eqref{ident1} by induction. 
The statement for $n=1$ is clear from the defining identity for $\hpq$. 
Assume true for $n=k$. 
For $n=k+1$ we have
\begin{align*}
  yx^{k+1} 
    &= (q^{k}x^ky + [k]_{p,q} x^{k-1} z)x \\
    &= q^{k}x^k(yx) + [k]_{p,q} x^{k-1} (z x) \\
    &= q^{k}x^k (q xy + z) + p^{-1} [k]_{p,q} x^k z \\
    &= q^{k+1}x^{k+1} y + (q^{k} + p^{-1} [k]_{p,q}) x^k z \\
    &= q^{k+1}x^{k+1} y + [k+1]_{p,q} x^k z \text{ by } \eqref{pqrel}.
\end{align*}
The proof of \eqref{ident2} is similar and left to the reader.
\end{proof}

In order to distinguish between the independent and
dependent parameter cases, we examine the prime ideals of $\hpq$.
Fix $p,q \in \kk^\times$ and for this remainder of this section let $H=\hpq$.

Define $\theta=(1-pq)yx-z$. 
Then $\theta$ is normal in $H$.
In particular, $\theta z = z\theta$, 
$\theta x = q x \theta$, and $\theta y = q\inv y \theta$. 
In $H/\theta H$ we have
\[ 0 	= \oy\ox-q\ox\oy-\oz 
		= \oy\ox - q\ox\oy - (1-pq)\oy\ox 
		= pq(\oy\ox-p\inv \ox\oy).\]
Thus, $H/\theta H \iso \qpp$. 
Clearly, $H/zH \iso \qp$.
Thus, the ideals $zH$ and $\theta H$ are prime in $H$. 
In the special case that $pq=1$, $zH=\theta H$.
Understanding these prime ideals allows us to determine when
$H$ is a \textit{polynomial identity (PI) ring}, 
that is, there exists a polynomial 
$f(\xi_1,\hdots,\xi_n) \in \ZZ\langle \xi_1,\hdots,\xi_n \rangle$
such that $f(h_1,\hdots,h_n)$ for all $h_i \in H$.

\begin{prop}
The algebra $H$ is PI if and only if $p$ and $q$ are
primitive roots of unity.
\end{prop}

\begin{proof}
Suppose $q$ is not a root of unity so that $\qp$ is not PI. 
Since $\qp$ is a quotient of $H$ and quotients of PI rings are again 
PI, then $H$ is not PI. Similarly for $p$.

Suppose $p$ and $q$ are primitive roots of unity of order 
$n$ and $m$, respectively. 
Once we show that 
$\kk[x^{mn},y^{mn},z^n] \subset \cnt(\hpq)$,
the result follows as $H$ is finitely 
generated as a module over its center \cite[Corollary 1.13 (i)]{mcrob}.
That $z^n$ commutes with $x$ and $y$ is clear and similarly that 
$x^{mn}$ and $y^{mn}$ commute with $z$. 
It remains to show that $x^{mn}$ and $y^{mn}$ commute with 
$y$ and $x$, respectively. 
Recall by \eqref{ident1} and \eqref{ident2} we have,
\begin{align*}
 yx^{mn} &= q^{mn}x^{mn}y + [mn]_{p,q} x^{mn-1}z = x^{mn}y, \\
 y^{mn}x &= q^{mn}xy^{mn} + [mn]_{p,q}zy^{k-1} = xy^{mn}.
\end{align*}
\end{proof}

By \cite[Theorem 4.9]{jorwells}, the ideals $zH$ and $\theta H$
are the only height one prime ideals in the independent parameter case.
We are interested in the dependent parameter case.
In this case, we recall the skew polynomial construction
$R[y;\sigma,\delta]$ with $R=\qpp=\kk[z][x;\alpha]$ from
Proposition \ref{prop.skpoly}.
Let $\mC$ be the Ore set of $R$ generated by $x$ and $z$.
By an abuse of notation, we denote the Ore set of $H$ 
generated by $x$ and $z$ also by $\mC$.
We denote $R\mC\inv$ and $H\mC\inv$ by $\mR$ and $\mH$, respectively.
When $p$ is not a root of unity, $\mR$ is a simple ring
\cite[Proposition 1.3]{mcpet}.

\begin{prop}
\label{dep.inner}
Suppose $p^r=q^s$ for some $r,s \in \ZZ$, $pq \neq 1$.
Then $\delta$ is an inner $\sigma$-derivation and 
$\sigma^r$ an inner automorphism on $\mR$.
\end{prop}

\begin{proof}
Let $\beta=(1-pq)\inv$ and $t=\beta zx\inv$. 
Then
\begin{align*}
  tz-\sigma(z)t 
  	&= (\beta zx\inv)z - \sigma(z)(\beta zx\inv) 
  	= \beta z (x\inv z - p\inv  zx\inv) = \delta(z)=0, \\
  tx-\sigma(x)t 
  	&= (\beta zx\inv)x - \sigma(x)(\beta zx\inv) 
  	= \beta (z - q(xz) x\inv)
    = \beta (z-pqz) = z =\delta(x).
\end{align*}

Let $a=q^r z^s x^r$. Then
\begin{align*}
	a\inv x a 
		&= x^{-r} z^{-s} x z^s x^r
		= p^s x x^{-r} z^{-s} z^s x^r = q^r x = \sigma^r(x), \\
	a\inv z a 
		&= x^{-r} z^{-s} z z^s x^r
		= p^{-r} x^{-r} z^{-s} z^s x^r z = p^{-r} z = \sigma^r(z).
\end{align*}
\end{proof}

\begin{remark}
An immediate consequence of Proposition \ref{dep.inner}
is that $\mH \iso \mR[a(y-t)^r]$.
There is no loss in assuming $\gcd(r,s)=1$ henceforth.
Thus, if $r,s < 0$ then there is no loss in assuming $r,s > 0$.
We will make this assumption, but this will leave the case
when $r$ and $s$ have different signs but $pq \neq 1$.
It should be possible to construct an analog of the 
Hayash-Weyl algebra in this case, but it would involve 
taking an appropriate localization of $H$ before factoring.
We will not deal with that case here.
\end{remark}

\begin{cor}
\label{cor.center}
Suppose $p^r=q^s$ for some $r,s \in \ZZ_+$, 
with $\gcd(r,s)=1$, $pq\neq 1$.
Assume $H$ is not PI.
\begin{itemize}
	\item[(1)] The center of $\mH$ (and hence, $H$)
is generated by $\Omega := (yx - p\inv xy)^r z^s$.
	\item[(2)] The prime ideals of $H$ which lie over $0$
	in $R$ are of the form $(\Omega - \alpha)H$ 
	for $\alpha \in \kk^\times$.
\end{itemize}
\end{cor}

\begin{proof}
By Proposition \ref{dep.inner}, $\mH = \mR[a(y-t)^r]$,
and so it is only left to verify that $a(y-t)^r$ is a 
scalar multiple of $\Omega$.
We first claim that $x(y-t) = \lambda \theta$ 
for some $\lambda \in \kk$.
\begin{align*}
  x(y-t) &= x(y-(1-pq)\inv zx\inv) = (1-pq)\inv \left[ (1-pq)(xy) - xzx\inv\right] \\
    &= (1-pq)\inv \left[ q\inv (1-pq) (yx-z) - pz\right] \\
    &= (1-pq)\inv q\inv \left[ (1-pq) (yx-z) - pqz\right] \\
    &= (1-pq)\inv q\inv \theta.
\end{align*}
It now follows inductively that 
$x^r(y-t)^r = \lambda^r q^{1-r}\theta^r$.
Since $z$ commutes with $\theta$, this implies that
$a(y-t)^r$ is a scalar multiple of $\theta^r z^s$.

Finally, we note that
	\[ \theta 
			= (1-pq)yx - z 
			= (1-pq)yx - (yx-qxy) 
			= -pqyx + qxy = -pq(yx-p\inv xy).
	\]
(1) now follows.

By \cite[Corollary 2.3]{lemat},
the primes lying over $0$ in $R$ are in 1-1 correspondence with the 
ideals of $\cnt(R)[\Omega]$, not including 
the ideal generated by $\Omega$.
Since $\cnt(R)=\kk$, (2) follows.
\end{proof}

Returning to the general case, 
we study isomorphisms between the $\hpq$ and 
the automorphism group of $\hpq$.

\begin{prop}
\label{hpq.isos}
If $\hpq \iso H_{p',q'}$ then $(p',q')$ is one of the following tuples: 
$(p,q)$; $(q,p)$; $(p\inv,q\inv)$; $(q\inv,p\inv)$.
\end{prop}

\begin{proof}
This follows more or less directly from \cite[Lemma 6.5]{bavjor}.
We elaborate briefly in the context of this problem.
The map $\hpq \rightarrow H_{p\inv,q\inv}$ 
is given by $x \mapsto y$, $y \mapsto x$, and $z \mapsto -qz$.
The map $\hpq \rightarrow H_{q,p}$ is given by
$x \mapsto i\sqrt{p}y$, $y \mapsto i\sqrt{p}x$, and 
$z \mapsto -\theta$.

One can now use the fact that any isomorphism either fixes the height one prime ideals $z\hpq$ and $\theta\hpq$,
inducing isomorphisms $\qp \rightarrow \mathcal{O}_{q'}(\kk^2)$ and
$\qpp \rightarrow \mathcal{O}_{p'}(\kk^2)$,
or else it switches them, inducing isomorphisms
$\qp \rightarrow \mathcal{O}_{p'}(\kk^2)$ and
$\qpp \rightarrow \mathcal{O}_{q'}(\kk^2)$.
Thus, these are the only isomorphisms up to composition.
\end{proof}

If $p=q$, then there is an involution $\tau$ of $\hpq$ which 
interchanges the height one prime ideals generated by $z$ and $\theta$.
In particular, this map is given by
$\tau(x)=\sqrt{q}y$, $\tau(y)=\sqrt{q}x$, and
$\tau(z)=\theta$.
It follows easily that $\tau(\theta)=pqz$.

\begin{prop}
\label{autos}
\[ \Aut(\hpq) = 
	\begin{cases}
		(\kk^\times)^2 \rtimes \{\tau\} & p = q^{\pm 1}, \\
		(\kk^\times)^2 & \text{otherwise}.
	\end{cases}\]
\end{prop}

\begin{proof}
The case $p=q$ is due to Alev and Dumas \cite[Proposition 2.3]{alev1}. The case $p=q\inv$ is similar and is, in fact, easier since
$z=\theta$ and every prime ideal contains $z$ even in the root
of unity case (Proposition \ref{inv.primes}).
The general case is due to Carvalho and Lopes 
\cite[Theorem 2.19]{carlopes}.
\end{proof}

\section{The inverse parameter case}
\label{sec.inv}

Throughout this section, fix $q \in \kk^\times$ and set $H=H_{q\inv,q}$,
$q \neq 1$.
When $q$ is a nonroot of unity,
every nonzero prime ideal of $H$ contains $z$ \cite[Proposition2.3]{kirksmall}.
Our approach will arrive at the same result but also include 
the root of unity case.

We retain notation from above.
In particular, the proof of Proposition \ref{dep.inner} shows that
$\sigma$ is an inner automorphism on $\mR$ in this case.
Then $\mH=\mR[ay;\hat{\delta}]$ where 
$a=(zx)\inv$ and $\hat{\delta}=a\delta$. 

\begin{prop}
\label{inv.primes}
Every prime ideal of $H$ contains $z$.
\end{prop}

\begin{proof}
By \cite[Lemma 3.21]{goodwa}, every prime ideal of $\mH$ intersects 
$\mR$ in a $\delta$-stable prime ideal. 
In case $q$ is not a root of unity, 
$\mR$ is a simple ring so we need only consider those 
ideals that intersect $\mR$ in zero. 
Let $P$ be a nonzero prime ideal of $\mH$ with $P \cap \mR = 0$. 
Such an ideal exists only if $\hat{\delta}$ is inner on $\mR$. 
We claim $\hat{\delta}$ is not inner on $\mR$.
Write $d = \sum \alpha_{ij} x^i z^j$. Then
\begin{align*}
x\inv &= az = \hat{\delta}(x) = dx-xd 
    = \sum \alpha_{ij} x^{i} (z^jx) - \sum \alpha_{ij} x^{i+1} z^j \\
    &= \sum \alpha_{ij} q^j x^{i+1} z^j - \sum \alpha_{ij} x^{i+1} z^j 
    = \sum \alpha_{ij} (q^j-1) x^{i+1} z^j.
\end{align*}
In order to have equality, we need that $(q^0-1)\alpha_{-1,0}=1$, 
but that is absurd.

In case $q$ is a primitive $n$th root of unity, $q \neq 1$,
it is left to check that the ideals $(x^n-a,z^n-b)$ of $\mR$ 
are not $\hat{\delta}$-invariant. 
Since $\hat{\delta}(x) = x\inv$, 
then an easy induction argument shows 
that $\hat{\delta}(x^n) = nx^{n-2}$. 
Thus $\hat{\delta}(x^n-a) = nx^{n-2}$ and so
these ideals are not $\hat{\delta}$-invariant  
when $n \neq 1,2$.

Suppose $n=2$ and let $I=(x^n-a,z^n-b)\mR$
for $a,b \in \kk^\times$. 
Then
  \[ \hat{\delta}(z(x^2-a)) 
  	= \hat{\delta}(z)(x^2-a) + z \hat{\delta}(x^2-a) 
  	= \hat{\delta}(z)(x^2-a) + 2z.\]
As $x^2-a \in I$, then $\hat{\delta}(z(x^2-a)) \in I$ 
if and only if $2z \in I$, a contradiction.
\end{proof}

\section{The independent parameter case}
\label{sec.indep}

Throughout this section, fix $p,q \in \kk^\times$ such that 
$q^r\neq p^s$ for all $r,s \in \ZZ$ and let $H=\hpq$.
Recall that, in this case, $H$ is primitive and $\cnt(H)=\kk$.
The Hayshi-Weyl algebra in this case was 
studied by Benkart \cite{benkweyl}.
It was also discussed in more generality by Futorny and Hartwig
\cite{mtwa} in the context of multiparameter twisted Weyl algebras.
They reference this as the \textit{generic case}.
Our interest here is to construct the algebra from the two-parameter 
analogs of the Heisenberg algebra as studied above.

Let $\sH$ be the localization of $H$ at powers of $z$ and $\theta$.
In this case $\sH$ is a simple ring by \cite[Theorem 4.9]{jorwells}.
Hence, there is no factor ring of $H$ analogous to the Hayashi-Weyl algebra.
To arrive at Benkart's algebras, we employ 
\textit{noncommutative dehomogenization} \cite{klr}.

Recall, the element $\theta=(1-pq)yx-z$ is normal in $H$
and hence also in $\sH$.
Define $\gamma \in \Aut(\sH)$ by 
$\gamma(h)=\theta\inv h \theta$ for all $h \in \sH$.
Thus, $\gamma$ is an inner automorphism given by 
$\gamma(x)=q\inv x$, $\gamma(y)=qy$, and $\gamma(z)=z$.

We form the skew polynomial ring $T=\sH[\omega;\gamma]$. 
Because $\gamma$ is inner, then $T=\sH[\theta\omega]$
and so the prime ideals of $T$ are 
generated by $\theta\omega-\alpha$ for $\alpha \in \kk$. 
Note that the ideal generated by $\theta\omega-1$ is equivalent to 
that generated by $\theta\inv-\omega$.

We will be interested in the algebra 
$\apq=T/(\theta\inv-\omega)T$. 
Thus, $\apq$ is presented as the $\kk$-algebra on generators 
$\{x,y,z^{\pm 1},\omega^{\pm 1}\}$ subject to the relations,
\begin{align*}
    zx=p\inv x z,& ~zy = p yz, \\
    \omega x = q\inv x \omega,& 	~\omega y = q y \omega, \\
    yx - qxy = z,& ~yx - p\inv xy = \omega\inv.
\end{align*}

\begin{prop}
\label{indep.props}
The algebra $\apq$ is a simple and noetherian with 
global dimension one and GK dimension two.
\end{prop}

\begin{proof}
That $\apq$ is a simple, noetherian domain follows 
from the above construction. 
By Proposition \ref{props1}, $\gldim H=\GK H=3$. 
Because $\gldim \sH \leq \gldim H < \infty$ and because $\sH$ is simple, 
then $\gldim \sH = 1$. 
Thus, $\gldim T = 2$. 
Since $\theta\inv-\omega$ is central in $T$, then $\gldim \apq \leq 1$. 
Since $T$ is not semisimple, then $\gldim \apq=1$. 

As $\sH$ is a localization of $H$, $\GK \sH=\GK H$. 
Thus, $\GK T=1+\GK\sH=4$.
Since $\theta\omega-1$ has degree two, 
then $\GK \apq \leq \GK T-2 = 2$. 
On the other hand, $z$ and $\omega$ generate
a commutative subalgebra of $\apq$ and so 
$2 = \GK \kk[z,w] \leq \GK \apq \leq 2$.
\end{proof}

\section{The dependent parameter case}
\label{sec.dep}

We now arrive at our primary area of focus,
the case wherein 
$p^r=q^s$ for some $r,s \in \ZZ_+$, 
with $\gcd(r,s)=1$, $pq\neq 1$.
Moreover, we assume $H=\hpq$ is not PI. 
We denote the algebra $H/(\Omega - 1)H$ by $\aprs$. 
Thus, $A_p(1,1)=\ap$ and in general $\aprs$ is the 
$\kk$-algebra on generators $x,y,z^{\pm 1}$ satisfying
\begin{align}
\label{aqn1} zx-p\inv xz = zy-pyz &= 0, \\
\label{aqn3} yx-q xy &= z, \\
\label{aqn4} (yx-p\inv xy)^r &= z^{-s}.
\end{align}
Set $w=(yx-p\inv xy)^{r-1} z^s$. 
Then $wx-q\inv xw = wy-qyw = wz-zw = 0$ and 
\eqref{aqn4} becomes 
\begin{align}
\label{aqn5} yx-p\inv xy = w\inv. 
\end{align}
Note that $w=z$ in the case $r=1$.

\begin{prop}
\label{aprs.gwa}
The algebra $\aprs$ is a GWA.
\end{prop}

\begin{proof}
Relations \eqref{aqn3} and \eqref{aqn5} may be rewritten as
\begin{align}
\label{aqn6} yx = \frac{z-pq w\inv}{1-pq}, \\
\label{aqn7} xy = p \frac{z-w\inv}{1-pq}.
\end{align}
Set $D=\kk[z^{\pm 1},w^{\pm 1}]$.
Define $\rho \in \Aut(D)$ by $\rho(z)=pz$ and $\rho(w)=q w$.
If $a=(1-pq)\inv (z - pq w^{-1})$, then it follows
by a direct check of \eqref{defn.gwa} that
$\aprs \iso D(\rho,a)$.
\end{proof}

As before, we denote by
$\mH$ (resp. $\mR$) the localization of $H$ (resp. $R$)
at powers of $x$ and $z$.

\begin{thm}
\label{simple}
The algebra $\aprs$ is a simple noetherian domain.
\end{thm}

\begin{proof}
That $\aprs$ is a noetherian domain follows from
Proposition \ref{aprs.gwa} and \cite[Proposition 1.3]{bavrep}.
The noetherian condition is also a direct consequence of 
$\aprs$ being a factor of $H$.
Simplicity follows from the maximality of the ideal $(\Omega-1)H$.
\end{proof}

\begin{prop}
\label{aprs.props}
Suppose $p$ is not a root of unity. 
Then $\gldim\aprs = 1$ and $\GK\aprs=2$.
\end{prop}

\begin{proof}
Since the ideal $(\Omega-1)H$ is generated by 
a central non-zero divisor, then $\GK\aprs \leq 2$. 
A ring of GK dimension one is necessarily PI.
Hence, if $p$ is not a root of unity, then $\GK\aprs=2$.

That $\gldim\aprs=1$ follows analogously to the Weyl algebra. 
If $r=1$, then this follows by \cite[Theorem 1.6]{bav.tensor}.
Observe that $\aprs$ is free as a $\kk[x]$-module (see below) 
and so $1 \leq \aprs$. 
Now define $B_1 = \aprs \tensor_{\kk[x]} \kk(x)$ and 
$B_2 = \aprs \tensor_{\kk[y]} \kk(y)$. 
Then $B_1 \iso B_2$ and moreover,
$B_1 \iso \mH/(\Omega - 1)\mH$, 

By Proposition \ref{dep.inner} $\mH \iso \mR[y-t;\sigma]$.
Since $\gldim\mR = 1$ \cite[Corollary 3.10]{mcpet}, 
then $\gldim\mH = 2$. 
The element $\Omega-1$ is central and regular, 
so $\gldim B_1 = 1$. 
Let $B=B_1 \oplus B_2$. 
Since $B$ is a faithfully free $\aprs$-module, then
$\gldim\aprs \leq \gldim B = \max\{\gldim B_1,\gldim B_2\} = 1$.
\end{proof}

\begin{prop}
\label{aprs.iso}
Let $s$ and $t$ be positive integers. 
If $A_p(1,s) \iso A_p(1,t)$, then $s=t$.
\end{prop}

\begin{proof}
Recall the construction of $\aprs$ as a GWA.
In this case, $w=z$ and so the ring $D$ reduces to $\kk[z^{\pm 1}]$.
Let $a_1(z) = (1-pq)\inv (z - pq z^{-s})$ and 
$a_2(z) = (1-pq)\inv (z - pq z^{-t})$. 
By \cite[Theorem 5.2]{bavjor}, 
$\kk[z^{\pm 1}](\rho,a_1) \iso \kk[z^{\pm 1}](\rho,a_2)$ 
if and only if there exists $\nu,\mu \in \kk^\times$ and 
a positive integer $\ell$ such that 
$a_2(z) = \nu z^\ell a_1(\mu z^{\pm 1})$. 
In the first case,
\begin{align*}
  (1-pq)\inv (z - pq z^{-t}) 
  	&= a_2(z) 
  	= \nu z^\ell a_1(\mu z) 
   	= \nu z^\ell (1-pq)\inv (\mu z - pq (\mu z)^{-s})\\
    &= (1-pq)\inv  (\nu \mu z^{\ell + 1} - pq \nu\mu^{-s}) z^{-s+\ell}.
\end{align*}
Comparing like terms, 
we either have $1=\ell + 1$ so $\ell=0$ and so $s=t$, 
or else $\ell+1=-t$ and $1 = -s+\ell$ so $t+s=-2$, which is absurd. 
The case of $z\inv$ is similar.
\end{proof}

Let $A_p^1(r,s)=\aprs$, then inductively we define 
$\apn = A_p^{n-1}(r,s) \tensor \aprs$ for $n \geq 1$. 
Explicitly, $\apn$ is the algebra on generators 
$\{x_i,y_i,z_i^{\pm 1},w_i^{\pm 1}\}$, $i=1,\hdots,n$, 
subject to the relations,
\begin{align*}
  [x_i,x_j] &= [y_i,y_j] = [z_i^{\pm 1},z_j^{\pm 1}] 
  		= [w_i^{\pm 1},w_j^{\pm 1}] 
  		= [z_i^{\pm 1},w_j^{\pm 1}] = 0, \\
  [x_i,y_j] &= 0 \text{ for } i \neq j, \\
  z_ix_j &= p^{-\delta_{ij}} x_jz_i, 	
  		~~~ z_iy_j = p^{\delta_{ij}} y_jz_i, \\
  w_ix_j &= q^{-\delta_{ij}} x_jw_i, 
  		~~~ w_iy_j = q^{\delta_{ij}} y_jw_i, \\
  y_ix_i &- q x_iy_i = z_i, 
  		~~~ y_ix_i - p\inv x_iy_i = w_i\inv, \\
  w_i &= (y_ix_i-p\inv x_iy_i)^{r-1} z_i^s.
\end{align*}
These relations imply,
\begin{align*}
	y_ix_i = \frac{z_i-pq w_i\inv}{1-pq} 
		\;\;\;\text{and}\;\;\;
	x_iy_i = p \frac{z_i-w_i\inv}{1-pq}.
\end{align*}
This gives the following identity,
\begin{align*}
  [y_ix_j,y_jx_i] &= \delta_{ij} (z_jw_i\inv - z_iw_j\inv).
\end{align*}

\begin{prop}
\label{vn.irred}
There exists a faithful irreducible representation of $\apn$.
\end{prop} 

\begin{proof}
For convenience, let $[m]=[m]_{p,q}$.
Let $V^n=\kk[\xi_1,\hdots,\xi_n]$.
For $\lambda = \xi_1^{m_1} \xi_2^{m_2} \cdots \xi_n^{m_n} \in V^n$,
the action of $\apn$ of $V^n$ is given by 
\begin{align*}  
  x_i.\lambda &= \xi_1^{m_1} \cdots 
  	\xi_i^{m_i+1} \cdots \xi_n^{m_n}, ~~
  y_i.\lambda = [m_i] \xi_1^{m_1} \cdots 
  	\xi_i^{m_i-1} \cdots \xi_n^{m_n}, \\
  z_i^{\pm 1}.\lambda &= p^{\mp m_i} \lambda, ~~
  w_i^{\pm 1}.\lambda = q^{\mp m_i}.\lambda.
\end{align*} 
Checking that this action satisfies the above relations is 
an easy exercise using the relations \eqref{pqrel}. 

Let $V'$ be a subrepresentation of $V^n$ 
and let $u \in V'$ be nonzero. 
Let $\lambda$ be the maximal monomial in $u$ according to 
the lexicographic ordering on $\ZZ^n$.
Let $\lambda$ be as above. Then
\begin{align*}
  (y_1^{m_1} \cdots y_n^{m_n}).\lambda 
  		&= [m_1]! [m_2]! \cdots [m_n]! \cdot 1.
\end{align*}
Since $p$ is not a root of unity, 
then no $[m_i]$ vanishes and so $1 \in V'$. 
\end{proof}

\begin{prop}
\label{apn.props}
The algebra $\apn$ is a simple noetherian domain with trivial center, 
global dimension $n$ and GK dimension $2n$.
\end{prop}

\begin{proof}
The statement on simplicity follows from the tensor 
product construction \cite[Theorem 1.7.27]{rowen}.
By Proposition \ref{vn.irred}, 
the representation $V^n$ is simple and faithful.
Hence, $\apn$ is primitive. 
Thus, by \cite[Proposition 3.2]{kirkkuz}, $\cnt(\apn)=\kk$.

By Proposition \ref{aprs.props}, $\GK \aprs =2$.
Hence, by \cite[Lemma 8.2.4]{mcpet},
	\[ \GK \apn = \GK (A_p^{n-1}(r,s) \tensor \aprs) 
		= \GK A_p^{n-1}(r,s)  + \GK \aprs.\]
The result on GK dimension now follows by induction.

By Proposition \ref{aprs.gwa}, $\aprs$ is a GWA.
Hence, $\apn$ is a tensor product of GWAs and is itself a GWA.
That $\apn$ is (left and right) noetherian now follows from 
\cite[Proposition 7]{bav.gwa}. 
The result on global dimension now follows from \cite[Lemma 1.1]{bav.tensor}.
\end{proof}

\section{Two-parameter quantum Virasoro algebra}
\label{sec.vira}

The (centerless) Virasoro Lie algebra is generated by
$\{x_i\}_{i \in \ZZ}$ with relations $[x_n,x_m]=(m-n)x_{n+m}$.
We will be most interested in the Virasoro enveloping algebra,
which we denote by $\mathcal{V}$.
Following \cite{kirksmall}, 
the q-analog of the Virasoro algebra $\vq$ is the $\kk$-algebra
generated by $\{y_i\}_{i \in \ZZ}$ with relations
	\[ q^{m-n}y_ny_m - q^{n-m}y_my_n = [m-n]_q y_{m+n}.\]
Kirkman and Small showed that $\vq$
appears as a subalgebra of the quotient division ring of $\aq$. 
Chakrabarti and Jagannathan have constructed a two-parameter analog
of the Virasoro algebra \cite{chakjag}, which we denote as $\vpq$.
This is the $\kk$-algebra on generators 
$\{L_n\}_{n \in \ZZ}$ with relations
\begin{align}
\label{vira1}  p^{n-m}L_nL_m - q^{m-n}L_mL_n &= [m-n]_{p,q}L_{m+n}.
\end{align}
Since we will focus only on the two-parameter case in this section,
there should be no confusion in setting $[k]=[k]_{p,q}$.

Fix $p,q$ and let $Q$ be the quotient ring of $\apq$
in the independent parameter case,
or that of $\aprs$ in the dependent parameter case.
We claim that $\vpq$ appears as a subalgebra of $Q$.
Once shown, we will be able to analyze the representation 
theory of $\vpq$ in a manner analogous to $\vq$.

\begin{prop}
The elements $L_n = z\inv x^{n+1} y \in Q$ generate a copy
of $\vpq$ in $Q$.
\end{prop}

\begin{proof}
First we consider the commutation rule for $L_n$ and $L_m$.
By \eqref{ident1}
\begin{align*}
	L_nL_m &= (z\inv x^{n+1} y)(z\inv x^{m+1} y) \\
		&= q^{-(n+1)}z\inv (yx^{n+1} - [n+1]x^n z)z\inv x^{m+1} y \\
		&= q^{-(n+1)}z\inv \left(p^{m+1}yx^{n+1}x^{m+1}z\inv 
				- [n+1]x^{n+m+1}\right) y \\
		&= q^{-(n+1)}z\inv \left(p^{m+1}(q^{m+1}x^{m+1}y 
				+ [m+1]x^mz)(p^{-(n+1)}z\inv x^{n+1}) - [n+1]x^{n+m+1}\right) y \\
		&= q^{m-n}p^{m-n}L_mL_n + q^{-(n+1)} (p^{m-n}[m+1]-[n+1])L_{n+m}.
\end{align*}
Hence, 
\begin{align*}
	p^{n-m}L_nL_m - q^{m-n}L_mL_n 
		&= q^{-(n+1)} ([m+1]-p^{n-m}[n+1])L_{n+m} \\
		&= (q-p\inv)\inv q^{-(n+1)} \left(q^{m+1} - p^{n-m}q^{n+1})\right)L_{n+m} \\
		&= (q-p\inv)\inv \left(q^{m-n} - p^{n-m}\right)L_{n+m} \\
		&= [m-n]L_{n+m}.
\end{align*}
\end{proof}

Let $B'$ be the subring of $Q$ generated by the elements 
$x$ and $z\inv y$.
Then 
	\[ (z\inv y)x = z\inv (qxy + z) = pqx(z\inv y) + 1.\]
Thus, $B' \iso A_1^{pq}(\kk)$ \eqref{qwa.def}.
The elements $\{x^i\}_{i \in \NN}$ form an Ore set in $B'$ and so 
$B=B'[x\inv]$ is noetherian.
Let $V$ be the subring of $B$ generated by $\{x^n(z\inv y)\}$. 

Recall that the \textit{idealizer} of a left ideal $J$ in a ring $R$ is 
defined as $\bI_R(J)=\{r \in R \mid Jr \subset J\}$, 
so that $\bI(J)$ is a two-sided ideal.

\begin{lem}
Let $I=B(z\inv y)$. 
In $B$ we have $V=\kk \oplus I = \bI_B(I)$.
\end{lem}

\begin{proof}
The first equality is clear by the commutation rule 
$x^nz\inv = p^{-n}z\inv x^n$ in $Q$. 
On the other hand, the generators of $V$ lie in $I$ so 
$V \subset \bI_B(I)$. 
Let $\mu=\sum \alpha_{ij} x^i (z\inv y)^j \in \bI_B(I)$,
then $I\mu \subset I$ and by \eqref{ident1},
\begin{align*}
I\mu  
	&= B\left[(z\inv y)\sum \alpha_{ij} x^i (z\inv y)^j\right] \\
	&= B\left[z\inv \sum \alpha_{ij}(q^i x^i y 
		+ [i]x^{i-1}z)(z\inv y)^j\right] \\
	&= B\left[\sum \alpha_{ij} q^i p^i x^i (z\inv y)^{j+1} 
		+ p^{i-1}[i]x^{i-1}(z\inv y)^j\right].
\end{align*}
The first term in the sum is an element of $I$ but the second term 
will lie in $I$ only if $j\geq 1$, so $\mu \in V$.
Thus, $\bI_B(I) \subset V$.
\end{proof}

\begin{lem}
The left ideal $I$ is maximal and generative.
\end{lem}

\begin{proof}
We claim $IB = B$, so that $I$ is generative. 
The inclusion $IB \subset B$ is clear. 
For the reverse inclusion, 
note that $(z\inv y)x, pq x(z\inv y) \in IB$. 
The defining relation $yx-qxy=z$ is equivalent to 
$yx\inv - q\inv x\inv y = -pq\inv x^{-2} z$ in $B$, 
and so $1 \in IB$.

For maximality, it suffices to show that the left module 
$M=B/I$ is simple. 
Let $v_k = x^k + I$. 
Since $\{x^i,(z\inv y)^j \mid i \in \ZZ,j \in \ZZ_+\}$ 
is a basis for $B$, 
then the set $\{v_k \mid k \in \ZZ\}$ is a basis for $M$. 
This gives the following relations in $M$,
\begin{align*}
  x^i v_k &= v_{k+i} \text{ for } i,k \in \mathbb{Z}, \\
  (z\inv y) v_k &= 
  	\begin{cases}
  		p^{k-1} [k] v_{k-1} & k \geq 1, \\ 
  		0 & k=0, \\ 
  		- p^{k+1} q\inv[-k] v_{k-1} & k \leq 1.  
  	\end{cases}
\end{align*}
The relation for the $x^i$ is clear. 
Suppose $k \geq 0$. Then
\begin{align*}
  (z\inv y) v_k 
  	&= z\inv (yx^k) + I 
  	= z\inv (q^k x^k y + [k]x^{k-1} z) + I
    = q^k (z\inv x^k) y + [k] (z\inv x^{k-1})z + I \\
    &= (pq)^k x^k (z\inv y) + p^{k-1}[k] x^{k-1} + I
    = p^{k-1}[k] v_{k-1}.
\end{align*}
If $k=0$, then the fact that $(z\inv y) \in I$ implies the relation. 
Suppose $k \leq 0$, then
\begin{align*}
  	(z\inv y) v_k 
  	&= z\inv (yx^k) + I 
  	= z\inv (q^k x^k y - (pq\inv)[-k]x^{k-1} z) + I \\
    &= q^k (z\inv x^k) y - (pq\inv) [-k] (z\inv x^{k-1})z + I \\
    &= (p\inv q)^k x^k (z\inv y) - p^{k+1} q\inv[-k] x^{k-1} + I
    = - p^{k+1} q\inv[-k] v_{k-1}.
\end{align*}
Suppose $N \subset M$ is a nonzero submodule and let 
$f = \sum \alpha_k v_k \in N$. 
Multiplication by a sufficiently large power of $x$ ensures that 
$\alpha_k = 0$ for all $k < 0$. 
Hence, there is no loss in assuming this holds for $f$. 
Since we may do this for any element in $N$, 
there is no loss in assuming that $f$ has minimal degree in $N$, 
that is, $\max\{k \mid \alpha_k \neq 0\}$ is minimal amongst elements of $N$. But then $(z\inv y)f$ has lower degree, 
a contradiction, so $v_0 \in N$ and so $N=M$.
\end{proof}

A left ideal $J$ in a ring $R$ is said to be \textit{isomaximal} 
if $R/J$ is the finite direct sum of isomorphic simple modules. 
A maximal left ideal is necessarily isomaximal.

\begin{thm}\cite[Theorem 5.5.5]{mcrob} 
Let $J$ be a generative isomaximal left ideal in 
a ring $R$ and $X$ a simple $R$-module.
\begin{enumerate}
  \item If $\Hom(R/J,X)=0$, then $X$ is a simple $\mathbb{I}_R(J)$-module. 
  \item If $M \iso R/X$ with $J \subset X$, then $X$ has a unique composition of length two over $\mathbb{I}_R(J)$, given by $R \supset \mathbb{I}_R(J) + M \supset M$.
\end{enumerate}
\end{thm}

Our final result is analogous to \cite[Corollary 6]{deansmall} 
and \cite[Proposition 1.8]{kirksmall},
and follows directly from the above discussion.

\begin{prop}
\label{vpq.reps}
Let $M\neq B/By$ be a simple $B$-module. 
Then $M$ is a simple $\vpq$ module. 
If $M=B/By$, then $B \supset \vpq + By \supset By$ is a composition series for $M$ when regarded as a $\vpq$ module.
\end{prop}

\section*{Acknowledgements} 
The author is indebted to 
Georgia Benkart, Bryan Bischof, David Jordan, and Lance Small 
for helpful discussions at various stages of this project.

\bibliography{biblio}{}

\begin{thebibliography}{10}

\bibitem{alev1}
J.~Alev and F.~Dumas.
\newblock Rigidit\'e des plongements des quotients primitifs minimaux de
  {$U_q({\rm sl}(2))$} dans l'alg\`ebre quantique de {W}eyl-{H}ayashi.
\newblock {\em Nagoya Math. J.}, 143:119--146, 1996.

\bibitem{artsch}
Michael Artin and William~F. Schelter.
\newblock Graded algebras of global dimension {$3$}.
\newblock {\em Adv. in Math.}, 66(2):171--216, 1987.

\bibitem{bavjor}
V.~V. Bavula and D.~A. Jordan.
\newblock Isomorphism problems and groups of automorphisms for generalized
  {Weyl} algebras.
\newblock {\em Trans. Amer. Math. Soc.}, 353(2):769--794, 2001.

\bibitem{bav.gwa}
Vladimir Bavula.
\newblock Generalized {Weyl} algebras, kernel and tensor-simple algebras, their
  simple modules.
\newblock In {\em Representations of algebras ({Ottawa}, {ON}, 1992)},
  volume~14 of {\em CMS Conf. Proc.}, pages 83--107. Amer. Math. Soc.,
  Providence, RI, 1993.

\bibitem{bav.tensor}
Vladimir Bavula.
\newblock Tensor homological minimal algebras, global dimension of the tensor
  product of algebras and of generalized {Weyl} algebras.
\newblock {\em Bull. Sci. Math.}, 120(3):293--335, 1996.

\bibitem{bavrep}
Vladimir Bavula and Viktor Bekkert.
\newblock Indecomposable representations of generalized {Weyl} algebras.
\newblock {\em Comm. Algebra}, 28(11):5067--5100, 2000.

\bibitem{benkweyl}
Georgia Benkart.
\newblock Multiparameter {W}eyl algebras.
\newblock {\em arXiv:1306.0485}, 2013.

\bibitem{benkroby}
Georgia Benkart and Tom Roby.
\newblock Down-up algebras.
\newblock {\em J. Algebra}, 209(1):305--344, 1998.

\bibitem{benkspoon}
Georgia Benkart and Sarah Witherspoon.
\newblock Two-parameter quantum groups and {D}rinfel'd doubles.
\newblock {\em Algebr. Represent. Theory}, 7(3):261--286, 2004.

\bibitem{diamond}
George~M. Bergman.
\newblock The diamond lemma for ring theory.
\newblock {\em Adv. in Math.}, 29(2):178--218, 1978.

\bibitem{carlopes}
Paula A. A.~B. Carvalho and Samuel~A. Lopes.
\newblock Automorphisms of generalized down-up algebras.
\newblock {\em Comm. Algebra}, 37(5):1622--1646, 2009.

\bibitem{chakjag}
R~Chakrabarti and R~Jagannathan.
\newblock A $(p,q)$-oscillator realization of two-parameter quantum algebras.
\newblock {\em J. Phys A: Math Gen}, 24:L711--L718, 1991.

\bibitem{deansmall}
C.~Dean and L.~W. Small.
\newblock Ring theoretic aspects of the {V}irasoro algebra.
\newblock {\em Comm. Algebra}, 18(5):1425--1431, 1990.

\bibitem{mtwa}
Vyacheslav Futorny and Jonas~T. Hartwig.
\newblock Multiparameter twisted {Weyl} algebras.
\newblock {\em J. Algebra}, 357:69--93, 2012.

\bibitem{goodwa}
K.~R. Goodearl and R.~B. Warfield, Jr.
\newblock {\em An introduction to noncommutative {N}oetherian rings}, volume~61
  of {\em London Mathematical Society Student Texts}.
\newblock Cambridge University Press, Cambridge, second edition, 2004.

\bibitem{hayashi}
Takahiro Hayashi.
\newblock {$q$}-analogues of {C}lifford and {W}eyl algebras---spinor and
  oscillator representations of quantum enveloping algebras.
\newblock {\em Comm. Math. Phys.}, 127(1):129--144, 1990.

\bibitem{jordan}
David~A. Jordan.
\newblock Height one prime ideals of certain iterated skew polynomial rings.
\newblock {\em Math. Proc. Cambridge Philos. Soc.}, 114(3):407--425, 1993.

\bibitem{jorwells}
David~A. Jordan and Imogen~E. Wells.
\newblock Simple ambiskew polynomial rings.
\newblock {\em J. Algebra}, 382:46--70, 2013.

\bibitem{klr}
A.~C. Kelly, T.~H. Lenagan, and L.~Rigal.
\newblock Ring theoretic properties of quantum {G}rassmannians.
\newblock {\em J. Algebra Appl.}, 3(1):9--30, 2004.

\bibitem{kirkkuz}
Ellen Kirkman and James Kuzmanovich.
\newblock Primitivity of {N}oetherian down-up algebras.
\newblock {\em Comm. Algebra}, 28(6):2983--2997, 2000.

\bibitem{kirkmus}
Ellen~E. Kirkman and Ian~M. Musson.
\newblock Hopf down-up algebras.
\newblock {\em J. Algebra}, 262(1):42--53, 2003.

\bibitem{kirksmall}
Ellen~E. Kirkman and Lance~W. Small.
\newblock {$q$}-analogs of harmonic oscillators and related rings.
\newblock {\em Israel J. Math.}, 81(1-2):111--127, 1993.

\bibitem{lemat}
Andr{\'e} Leroy and Jerzy Matczuk.
\newblock Prime ideals of {O}re extensions.
\newblock {\em Comm. Algebra}, 19(7):1893--1907, 1991.

\bibitem{leva}
Thierry Levasseur.
\newblock Some properties of noncommutative regular graded rings.
\newblock {\em Glasgow Math. J.}, 34(3):277--300, 1992.

\bibitem{mcpet}
J.~C. McConnell and J.~J. Pettit.
\newblock Crossed products and multiplicative analogues of {W}eyl algebras.
\newblock {\em J. London Math. Soc. (2)}, 38(1):47--55, 1988.

\bibitem{mcrob}
J.~C. McConnell and J.~C. Robson.
\newblock {\em Noncommutative {N}oetherian rings}, volume~30 of {\em Graduate
  Studies in Mathematics}.
\newblock American Mathematical Society, Providence, RI, revised edition, 2001.
\newblock With the cooperation of L. W. Small.

\bibitem{rowen}
Louis~H. Rowen.
\newblock {\em Ring theory}.
\newblock Academic Press Inc., Boston, MA, student edition, 1991.

\bibitem{zhang}
J.~J. Zhang.
\newblock Twisted graded algebras and equivalences of graded categories.
\newblock {\em Proc. London Math. Soc. (3)}, 72(2):281--311, 1996.

\end{thebibliography}
\bibliographystyle{plain}

\end{document}